\documentclass[11pt]{article} 
\usepackage{amssymb}
\usepackage{amsmath}
\usepackage{amsthm}
\usepackage{mathtools}
\newtheorem{prop}{\bfseries Proposition}
\newtheorem{lem}{\bfseries Lemma}
\newtheorem{thm}{\bfseries Theorem}

\newtheorem{defn}{\bfseries Definition}

\renewcommand{\=}[1]{\stackrel{(\ref{#1})}{=}}
\begin{document}

\begin{center}
{\bf\large\sc 
Totally positive matrices and dilogarithm identities 
} 

\vspace*{6mm}
{  Andrei Bytsko \ and \ Alexander Volkov }

\end{center}
\vspace*{1mm}

\begin{abstract}
\noindent

We show that two involutions on the variety $N_n^+$ of upper triangular totally 
positive matrices are related, on the one hand,
to the tetrahedron equation and, on the other hand, to the action 
of the symmetric group $S_3$ on some subvariety of $N_n^+$
and on the set of certain functions on $N_n^+$.
Using these involutions, we obtain a family of dilogarithm identities 
involving minors of totally positive matrices.
These identities admit a form manifestly invariant under the action
of the symmetric group $S_3$.
 
\end{abstract}

{\small
}

\section{Introduction}

The quantum dilogarithm identities related to the quantum 
tetrahedron equation constructed in \cite{BV1} 
  possess an $S_3$ symmetry
in the sense that they are invariant under a certain automorphism of 
order three and three involutive antiautomorphisms of the related 
quantum torus algebra.  According to
\cite{FK1}, the quasi-classical limit of a quantum dilogarithm identity
yields an identity for the Rogers dilogarithm function. 
In the present article, we will construct a family of
identities for the Rogers dilogarithm that are closely related to
the quantum dilogarithm identities mentioned above. In particular,
they will be shown to possess an $S_3$ symmetry in a similar sense.
We will not derive the Rogers dilogarithm identities from their 
quantum counterparts via the quasi-classical limit. Rather,  we will 
establish them by introducing certain involutions on the variety of upper triangular 
totally positive matrices. These involutions are closely related to the 
BFZ twist -- a transformation of totally positive matrices
considered in \cite{BFZ} and, as we will show below, they are also
related to the tetrahedron equation and the action of the 
symmetric group~$S_3$.

\subsection{Rogers dilogarithm}
 
Recall that, for $x \in (0,1)$, the Rogers dilogarithm is defined   by 
\begin{align}
L (x)  = Li_2(x) + \frac{1}{2} \log x\, \log(1-x) .
\end{align} 
In the present paper, we will use the following 
closely related function:
\begin{align*}
 l(x)  = \frac{6}{\pi^2}\, L\Bigl(\frac{x}{1+x}\Bigr) ,
\qquad\qquad \, x>0 \,.
\end{align*} 
The well-known functional relations 
for the Rogers dilogarithm (see, e.g. \cite{Kir}) 
acquire the following form for $l(x)$:
\begin{eqnarray} 
\label{rd1}
&\displaystyle 
 l(x) + l\bigl( 1/x \bigr) = 1 \,, &  \\[2mm] 
\label{pent} 
&\displaystyle  
 l(x) + l(y) \ =\ l\Bigl(\frac{x}{1+y}\Bigr) 
+ l\Bigl(\frac{x y}{1+x+y}\Bigr) + l\Bigl(\frac{y}{1+x}\Bigr)  \,. &
\end{eqnarray}
The latter identity is usually referred to as the pentagon relation.
It has an obvious ${\mathbb Z}_2$ symmetry in the sense that (\ref{pent})
is invariant under the change of variables $x \leftrightarrow y$.

\subsection{Totally positive matrices}

Let $B_n$, $n \geq 2$, denote the group of $n\,{\times}\,n$ 
  real invertible upper triangular matrices. Let $N_n$ be the unipotent
subgroup of $B_n$.  
Adopting the terminology of \cite{BFZ}, we define  
totally positive (with respect to $B_n$ or $N_n$) matrices as follows.

\begin{defn}
$G \in B_n$ is said to be a {\em totally positive} matrix if every minor of $G$ 
that does not 
vanish identically is positive. $B^+_n$ and $N^+_n$ are, respectively,  
the subvarieties of $B_n$ and $N_n$ consisting of totally positive matrices.  
\end{defn}

For $x \in {\mathbb R}$, we set $J_k(x) = \exp ( x e_k) \in N_n$, where  
$e_k$ is the matrix unit such that 
$ (e_k)_{ij} = \delta_{ik} \delta_{j,k+1}$.  

\begin{lem}[\cite{BFZ}, Prop.\,1.7]\label{Jlem}  
If $M \in N_n$ is given by
\begin{align}\label{Mjac}
  M = {\mathbb J}_{1} \ldots {\mathbb J}_{n-1} \,, 
   \qquad \text{where} \quad  
   {\mathbb J}_{k} = J_k(x_{1,k+1}) \ldots J_1(x_{k,k+1}) \,,
\end{align}
and $x_{ij}$, $1 \leq i < j \leq n$, are positive real numbers, then $M$
is totally positive. Moreover,   every $M \in N^+_n$ admits 
 a unique representation of the form (\ref{Mjac}). 
\end{lem} 
The order of the factors  $J_k$ in (\ref{Mjac}) corresponds
to the lexicographically minimal reduced word  for the maximal length element 
$w_0$ of the symmetric group $S_{n}$.  
Factorizations of totally positive matrices
 corresponding to arbitrary minimal reduced words have been studied in~\cite{BFZ}.
We will refer to $x_{ij}$ in (\ref{Mjac})  as the Jacobi coordinates of $M$.

Note that every $G \in B_n^+$ is uniquely represented as 
$G = M\, \Lambda$, where   $M \in N_n^+$ and $\Lambda$ is a
diagonal matrix with positive diagonal entries. 

\section{Relation of a dilogarithm identity to 4$\times$4 positive matrices }

\subsection{An $S_3$-invariant dilogarithm identity}

Consider the following function in three variables:
\begin{align}\label{defF} 
{}& F(x,y,z)  = 
   l\Bigl(\frac{x}{1+y}\Bigr)   
 +  l\Bigl(\frac{(1+x+y)z}{(1+x)(1+y)}\Bigr) + 
 l\Bigl(\frac{xy}{(1+x+y)(1+z)}\Bigr)+ 
 l\Bigl(\frac{y}{1+x} \Bigr)  .
\end{align} 
Note that the arguments of the dilogarithms in (\ref{defF}) are positive
 if $x,y,z>0$ or if $x,y,z<-1$.

Our starting point is the observation that this function
 is $S_3$-invariant in the following sense.
\begin{prop}\label{FFFF}
$F(x,y,z)$ is invariant under any permutation of its arguments,
\begin{align}\label{Fcycl1} 
  F(x,y,z) = F(y,x,z) = F(x,z,y) = F(z,y,x) = F(y,z,x) = F(z,x,y) \,.
\end{align}   
\end{prop}

\begin{proof}
Setting $t = \frac{x y}{1+x+y}$ and using  the pentagon
relation (\ref{pent}) twice, we obtain
\begin{align*} 
{}& l(x) + l(y) + l(z) - l\Bigl(\frac{xyz}{1+x+y+z+xy+xz+yz}\Bigr) \\
{}& = l(x) + l(y) + l(z) - l\Bigl(\frac{tz}{1+t+z}\Bigr)  \\
{}& \={pent}  l\Bigl(\frac{x}{1+y}\Bigr) + l\Bigl(\frac{y}{1+x}\Bigr)
+ l(t) + l(z) - l\Bigl(\frac{tz}{1+t+z}\Bigr) \\
{}& \={pent} l\Bigl(\frac{x}{1+y}\Bigr) + l\Bigl(\frac{y}{1+x}\Bigr)
+ l\Bigl(\frac{t}{1+z}\Bigr) + l\Bigl(\frac{z}{1+t}\Bigr)
= F(x,y,z) \,.
\end{align*}
Whence it is obvious that $F(x,y,z)$ is totally symmetric in $x,y,z$.
\end{proof}

Let us note that, in the limit $z \to 0$, identity $F(x,y,z) =F(x,z,y)$ turns into
the pentagon identity~(\ref{pent}).
 
 \subsection{From the dilogarithm identity to totally positive matrices}
 
Remarkably,  the arguments of the dilogarithms in the identities (\ref{Fcycl1})
can be expressed as simple functions (ratios of products) of certain minors 
of a totally positive matrix and its images under two involutive transformations.
 
Hereafter, we will use the notation $[a,b]$ for 
 the set $\{ i \in {\mathbb N} \,|\, a \leq i \leq b\}$.
 
 Let $G \in B_n$.
Let $I$ and $J$ be some subsets of the set $[1,n]$.  
Denote their cardinalities by $|I|$ and $|J|$.
We will use the notation $\Delta_I(G)$ for the right flag minor of $G$  
 corresponding to intersection of the last $|I|$ columns of $G$ 
 with the rows labeled by the set $I$. Analogously, 
$\Delta^J(G)$ will stand for the upper flag minor of $G$ 
corresponding to intersection of the  first $|J|$ rows of $G$ 
with the columns labeled by the set $J$.
 We will omit the dependence of $\Delta_I$ and  $\Delta^J$
on $G$ when it does not lead to a confusion. For convenience,
we set $\Delta_\emptyset=\Delta^\emptyset =1$.

Given a set $I \subset [1,n]$, let $\bar{I}$ stand for the set such that
$\{a\} \in \bar{I}$ iff $\{n+1-a\} \in I$.

\begin{lem}\label{BBc}
For every $G \in B^+_n$, there exist unique $\check{G}, \hat{G} \in B^+_n$
such that
\begin{align}\label{ggc}
   \Delta_I \bigl( \check{G} \bigr) = \Delta_{\bar{I}} (G) , \qquad
   \Delta^J \bigl( \hat{G} \bigr) = \Delta^{\bar{J}} (G)  
\end{align} 
for all $I,J \subset [1,n]$.
\end{lem}
\noindent
(The proof is deferred to the Appendix.)

Since every $G \in B^+_n$ is determined completely by the set of its 
flag minors $\Delta_I(G)$ or $\Delta^J(G)$  (cf.\ the proof of Lemma~\ref{BBc}),
we conclude that equations (\ref{ggc}) define two maps on $B^+_n$,
namely, $G \to \check{G}$ and $G \to \hat{G}$.  Taking into account that 
$\bar{{\bar{I}}}=I$, we infer that these maps are involutions, i.e. $\check{\check{G}} = G$ 
and $\hat{\hat{G}} = G$. 

Let us introduce a function ${\mathcal L}$ on $B^+_4$ 
 ($\Delta_I$ stands for $\Delta_I(G)$):
\begin{align}\label{LL1}
{}
 {\mathcal L}(G) &=
  l\Bigl(\frac{\Delta_{14} \Delta_{234}}{\Delta_{34}\Delta_{124} }\Bigr) +
l\Bigl(\frac{\Delta_{1} \Delta_{24}}{\Delta_{4}\Delta_{12} }\Bigr) +
l\Bigl(\frac{\Delta_{12} \Delta_{234}}{\Delta_{24}\Delta_{123} }\Bigr) + 
 l\Bigl(\frac{\Delta_{2} \Delta_{34}}{\Delta_{4}\Delta_{23} }\Bigr)  .
\end{align} 

\begin{prop}\label{DilG}
For every $G \in B^+_4$,  the following   dilogarithm identities hold:
\begin{align}\label{Lggg}
      {\mathcal L}\bigl( \check{G} \bigr) = 
     {\mathcal L}\bigl( \hat{G} \bigr) = 4 - {\mathcal L}(G).
\end{align} 
\end{prop}

\begin{proof} Let $G = M\,\Lambda$, where $\Lambda \in B_4$ is a
diagonal matrix with positive diagonal entries and 
$M \in N^+_4$ is given by
\begin{align}\label{Jac4}
M=J_1(x_{12}) J_2(x_{13}) J_1(x_{23}) 
J_3(x_{14}) J_2(x_{24}) J_1(x_{34}) .
\end{align} 
Set $\delta = x_{24} x_{34} - x_{12} x_{13}$.
In the generic case, i.e. when $\delta \neq 0$, we can introduce  
variables  $x,y,z$ related to the Jacobi
coordinates of $M$ in the following way: 
\begin{align}\label{xyz}
 \frac{z}{1+y} = \frac{x_{12}}{x_{23}} , \qquad 
 \frac{y}{1+x} = \frac{x_{13}}{x_{24}} , \qquad 
 \frac{x}{1+z} = \frac{x_{23}}{x_{34}} \,.
\end{align} 
Indeed, (\ref{xyz}) is a system of three linear equations in $x,y,z$ 
that has a solution if $\delta \neq 0$:
 $x=(x_{12} x_{13} + x_{12} x_{24} + x_{23} x_{24})/\delta$,
etc. Since the terms on the r.h.s. of (\ref{xyz}) are positive, it is
clear that if $x,y,z$ is a solution of this system, then either 
$x,y,z>0$ or $x,y,z < -1$.

For $\delta \neq 0$, a direct computation shows that the arguments of
the dilogarithms in (\ref{LL1}) match those in (\ref{defF}), e.g.
$\Delta_{14}(G) \Delta_{234}(G)/(\Delta_{34}(G)\Delta_{124}(G)) 
 = \frac{x}{1+y}$ etc. Similarly, the arguments of the 
dilogarithms in ${\mathcal L}\bigl( \check{G} \bigr) $ and
${\mathcal L}\bigl( \hat{G} \bigr)$ match (up to a permutation of
the terms in the sum), respectively, 
the reciprocals of the arguments of the 
dilogarithms in $F(z,y,x)$ and $F(x,z,y)$, e.g. 
$\Delta_{14}(\check{G}) \Delta_{234}(\check{G})/
(\Delta_{34}(\check{G})\Delta_{124}(\check{G})) 
=\frac{1+z}{y}$. 
Therefore, taking relation (\ref{rd1}) into account, we have 
\begin{align}\label{fgdel}
 {\mathcal L}(G) = F(x,y,z) , \qquad
 {\mathcal L}\bigl( \check{G} \bigr) = 4 -  F(z,y,x) , \qquad
  {\mathcal L}\bigl( \hat{G} \bigr) = 4 - F(x,z,y) . 
\end{align} 
Combining these equations with identities (\ref{Fcycl1}),
we obtain identities (\ref{Lggg}).

In the special case, i.e. when $\delta = 0$, a direct computation
shows that,  for ${\mathcal L}(G)$, ${\mathcal L}\bigl( \check{G} \bigr)$,
and ${\mathcal L}\bigl( \hat{G} \bigr)$,  the arguments of
the first and the fourth dilogarithms in (\ref{LL1}) are 
reciprocals of each other and so are  the arguments of
the second and the third dilogarithms. By (\ref{rd1}), this implies that,
if $\delta = 0$, we have ${\mathcal L}(G)={\mathcal L}\bigl( \check{G} \bigr)
= {\mathcal L}\bigl( \hat{G} \bigr)=2$ and thus identity (\ref{Lggg})
holds trivially in this case.
\end{proof}

\section{Totally positive matrices,  
tetrahedron equation, symmetric group $S_3$, and dilogarithm identities}

\subsection{Involutions $M'$ and $M''$ on $N_n^+$}

In order to extend identities (\ref{Lggg}) to the case of $n>4$,
we will introduce the counterparts of the maps 
$G \to \check{G}$ and $G \to \hat{G}$ for unipotent upper triangular 
matrices (note that if $M \in N^+_n$, then in general
$\check{M}$ and $\hat{M}$ are not unipotent). 

Let $P$  denote the $n\,{\times}\,n$ permutation
matrix corresponding to the maximal length element $w_0$ of the 
symmetric group $S_n$, i.e. $P_{ij}=\delta_{i+j,n+1}$.

\begin{lem}\label{GM}
For every $M \in N^+_n$, there exist  unique  $M', M'' \in N^+_n$ 
and a unique diagonal matrix $D_{\!\scriptscriptstyle M}$
such that  
\begin{align} \label{gauss}
  P M P = M' P D_{\!\scriptscriptstyle M} M'' .
\end{align} 
\end{lem}
\noindent
(Hereafter, the proofs of lemmas and most of propositions are deferred to the  Appendix). 

Equation (\ref{gauss}) defines two maps on $N^+_n$,
namely, $M \to M'$ and $M \to M''$. They are related to the maps 
$G \to \check{G}$ and $G \to \hat{G}$ on $B^+_n$ as follows: 

\begin{lem}\label{GMB}
Let $\Lambda  \in B_n$ be  a diagonal matrix with positive diagonal entries. \\
a) For $G = M \, \Lambda$, where $M \in N_n^+$, we have
\begin{align} \label{gc}
  \check{G} = M' \, \Lambda \, P \tilde{D}_{\!\scriptscriptstyle M} P .
\end{align}   
b) For $G =\Lambda\,  M$, where $M \in N_n^+$, we have
\begin{align} \label{gh}
  \hat{G} =  \Lambda \, \tilde{D}_{\!\scriptscriptstyle M} M'' .
\end{align} 
In (\ref{gc}) and (\ref{gh}), $\tilde{D}_{\!\scriptscriptstyle M} \in B_n$ is a
diagonal matrix such that 
$(\tilde{D}_{\!\scriptscriptstyle M})_{ii} = |(D_{\!\scriptscriptstyle M})_{ii}|$.
\end{lem}

\begin{prop}\label{SG}
a) The mappings $M \to M'$ and $M \to M''$ are involutions, i.e. the relations
\begin{align} \label{invol}
  (M')' = M \,, \qquad  (M'')''  = M 
\end{align}
hold for every $M \in N^+_n$. \\[0.5mm]
b) If $M \in N_n^+$ is given by (\ref{Mjac}), then 
the Jacobi coordinates of $M', M''$ are given by
\begin{align}\label{xxm}
 x_{ij}(M') =  
  \frac{ \prod_{k=1}^{i-1} x_{k,n+i-j} (M)}%
 { \prod_{k=1}^{i} x_{k,n+1+i-j}(M) } , \qquad
  x_{ij}(M'') =  
 \frac{ \prod_{k=j+1}^{n} x_{j+1-i,k}(M) }%
{ \prod_{k=j}^{n} x_{j-i,k} (M)} \,.
\end{align} 
\end{prop}

Let us remark that we have defined the maps $M \to M'$, $M \to M''$  by
the relation (\ref{gauss}) that is similar in spirit to the relation 
$P z^t = v^t \, d \, u$, where $u,v,z \in N^+_n$, $d$ is a diagonal matrix,
$t$ stands for the matrix transposition.
The latter relation was used in \cite{BFZ} in order to define the map
$z \to \eta(z)=u$ (called the BFZ twist). It is not difficult to check how
these maps are related: 
$M' = P( \eta(M) )^t P$, $M'' = \eta(P M^t P)$. We would like 
to stress that, unlike the BFZ twist,  the maps $M \to M'$, $M \to M''$
are involutions. 
This fact will be very important below in the context of the $S_3$ symmetry.

\subsection{Involutions $\bar{M}'$ and $\bar{M}''$ and 
the tetrahedron equation}

Let $M \to \bar{M}$ denote the involution on $N_n^+$ such that
\begin{align}\label{Mbar}
    x_{ij}(\bar{M}) = \frac{1}{x_{ij}(M)} 
\end{align}
for all $1 \leq i < j \leq n$.
Equations (\ref{xxm}) imply that $(\bar{M})'= \overline{(M')}$ and 
$(\bar{M})''= \overline{(M'')}$, so that we can write simply
$M \to \bar{M}'$ and $M \to \bar{M}''$ for the corresponding mappings.

According to equation (\ref{Mjac}), matrix entries of $M \in N^+_n$ are functions in
$n(n-1)/2$ variables $x_{12},\ldots,x_{n-1,n}$. Consider the  
changes of these variables, $L_{abc}$ and $R_{abc}$, where 
$1 \leq a < b < c \leq n$, that affect only the variables $x_{ab}$,
 $x_{ac}$, and $x_{bc}$ and are given by
 \begin{align}
 \label{Labc}
 {}&    L_{abc}(x_{ab}) = x_{ac} , \ \ L_{abc}(x_{ac}) = x_{ab}, \ \ 
     L_{abc}(x_{bc}) = \frac{ x_{ac} x_{bc} }{ x_{ab} }, \\
\label{Rabc}     
{}&    R_{abc}(x_{ab}) = \frac{ x_{ab} x_{ac} }{ x_{bc} } , \ \ R_{abc}(x_{ac}) = x_{bc}, \ \ 
     R_{abc}(x_{bc}) = x_{ac} .
\end{align}

It follows from (\ref{xxm}) that for $M \in N_3^+$ (considered, according to (\ref{Mjac}),
as a matrix function in the variables $x_{12}$, $x_{13}$, and $x_{23}$) we have 
\begin{align}\label{MMlr}
      \bar{M}' = L_{123} (M) , \quad \bar{M}'' = R_{123} (M).
\end{align}

Furthermore, for $n=4$, we have the following statement:

\begin{prop}\label{TE} 
Let $M \in N_4^+$ be given by (\ref{Jac4}) and thus it is considered  
as a matrix function in the variables $x_{12}$, \ldots, $x_{34}$. Then
\begin{align}
\label{MMl1}
{}&  L_{123} ( L_{124} ( L_{134}( L_{234} (M)))) = 
  L_{234} ( L_{134} ( L_{124}( L_{123} (M)))) = \bar{M}' , \\
\label{MMr1}
{}& R_{123} ( R_{124} ( R_{134}( R_{234} (M)))) = 
  R_{234} (  R_{134} ( R_{124}( R_{123} (M)))) = \bar{M}'' .
\end{align}
\end{prop}

\begin{proof} A straightforward computation.
\end{proof}

The first equalities in (\ref{MMl1}) and (\ref{MMr1}) imply that 
the transformations $L_{abc}$ and $R_{abc}$ satisfy the tetrahedron 
equation. In the context of the quantum tetrahedron equation, 
this fact was used in \cite{KV2} and later in \cite{BV1}.
A comparison of the second equation in (\ref{xxm}) and equations (\ref{Rabc}) 
with, respectively,  the first equation in (A.35) and equation (42) in \cite{BV1} shows that
relations (\ref{MMl1}) and (\ref{MMr1}) can be generalized to the case of
arbitrary $n$ by replacing the composition of four $L$'s or $R$'s
by the lexicographically ordered composition of $n \choose 3$ $L$'s or $R$'s
(cf. equation (15) in~\cite{BV1}).

\subsection{Involutions $M'$ and $M''$ and the symmetric group $S_3$}

\begin{defn}
$\tilde{N}^+_n$ is the subvariety of $N^+_n$
consisting of matrices $M$ whose Jacobi coordinates satisfy the
following relations 
\begin{align} \label{con}
   \prod_{k=i+1}^n \frac{x_{ik} (M) }{x_{n+1-k,n+1-i} (M)} =
     \prod_{k=1}^{i-1}  \frac{x_{ki} (M)}{x_{n+1-i,n+1-k} (M)}
\end{align}
for all $1 \leq i < n/2$.
\end{defn}
Let us remark that, using formula (\ref{mab}), we can rewrite relation (\ref{con})
as a condition on the corner minors of $M$:\ 
$\Delta_{[1,i]} (M) \, \Delta_{[1,n+1-i]} (M) = \Delta_{[1,i-1]} (M) \, \Delta_{[1,n-i]} (M)$.

\begin{prop}\label{MS3}
a) $\tilde{N}^+_n$ is invariant under the mappings $M \to M'$, $M \to M''$. That is,
 if $M \in \tilde{N}^+_n$, then $M', M'' \in \tilde{N}^+_n$.\\[0.5mm]
b) For every $M \in \tilde{N}^+_n$, the following relation holds:
\begin{align} \label{Mrho}
  \bigl( (M')'' \bigr)' = \bigl( (M'')' \bigr )'' .
\end{align}
\end{prop}

Let $\sigma_1$ and $\sigma_2$ denote the generators of the symmetric group $S_3$
satisfying the relations $\sigma_1 \sigma_1 = \sigma_2 \sigma_2 = id$,
$\sigma_1 \sigma_2 \sigma_1 = \sigma_2 \sigma_1 \sigma_2$.
Proposition~\ref{MS3}, along with the part a) of Proposition~\ref{SG},  implies 
 that the action of the group $S_3$ on
$\tilde{N}^+_n$ given by 
\begin{equation}\label{s3nt}
\begin{aligned} 
{}    id(M)  &=M , \quad \sigma_1(M)=M', \quad \sigma_2(M)=M'', \quad
 (\sigma_2  \sigma_1)(M)=(M')'' , \\
{}  (\sigma_1   \sigma_2)(M)  &= (M'')'  , \quad 
 (\sigma_1  \sigma_2   \sigma_1)(M)=((M')'')' , \quad 
  (\sigma_2   \sigma_1  \sigma_2)(M)=((M'')')''
\end{aligned}
\end{equation}
is a homomorphism from $S_3$ to the automorphism group of $\tilde{N}^+_n$.  

Although relations (\ref{s3nt}) do not define a group action of $S_3$ on ${N}^+_n$
(since (\ref{Mrho}) does not hold in general if $M \in {N}^+_n$), we will not restrict our
 consideration  to the subvariety $\tilde{N}^+_n$ because  conditions  (\ref{con}) result in
a degeneration of the variables in the dilogarithm identities that we want to obtain. 
For instance, for $n=4$, condition (\ref{con}) reads $x_{12}  x_{13} = x_{24} x_{34}$,
in which case all the three functions in (\ref{Lggg}) become constants (cf.\ the
end of the proof of Proposition~\ref{DilG}).

Note that the variables $x,y,z$ in (\ref{xyz}) depend only
on certain ratios of the Jacobi coordinates $x_{ij}$. This observation helps
to reveal the role of the group $S_3$ for ${N}^+_n$. Consider 
   functions $y_{ij}, \bar{y}_{ij} : N^+_n \to {\mathbb R}$, $1 \leq i < j \leq n-1$,
  given by
\begin{align} \label{yyxx}
    y_{ij}(M) = \frac{x_{ij}(M)}{x_{i+1,j+1}(M)} , 
    \qquad  \bar{y}_{ij}(M) = y_{ij}(\bar{M}) ,
\end{align}
where the map $M \to \bar{M}$ was defined in (\ref{Mbar}).

\begin{prop}\label{SY}
The functions $y_{ij}, \bar{y}_{ij}$ have the following transformation properties:
\begin{eqnarray} \label{my1}
   \displaystyle  y_{ij}(M') =  \bar{y}_{i,n+i-j}(M) , && \quad    
    y_{ij}(M'') = \bar{y}_{j-i,j}(M) ,  \\[1mm]
\label{my2}    
     y_{ij}\Bigl(\bigl( (M')'' \bigr)' \Bigr) =  y_{ij}\Bigl(\bigl( (M'')' \bigr )'' \Bigr) , && \quad
   \bar{y}_{ij}\Bigl(\bigl( (M')'' \bigr)' \Bigr) =  \bar{y}_{ij}\Bigl(\bigl( (M'')' \bigr )'' \Bigr) . 
\end{eqnarray}

\end{prop}  
As a consequence, we can define a group action of $S_3$ on the set of functions
$y_{ij}, \bar{y}_{ij}$: 
$\sigma_1(y_{ij}) = \bar{y}_{i,n+i-j}$, $\sigma_2(y_{ij})  = \bar{y}_{j-i,j}$, 
$\sigma_1(\bar{y}_{ij}) = y_{i,n+i-j}$, $\sigma_2(\bar{y}_{ij}) = y_{j-i,j}$ etc.
This group action, along with Lemma~\ref{GMB}, clarifies, for $n=4$, the link between  
identities (\ref{Lggg}) and (\ref{Fcycl1})
and suggests a way of generalizing identities (\ref{Lggg}) to the case of $n>4$.

\subsection{Involutions $M'$ and $M''$ and dilogarithm identities}

Let us introduce 
four sets of functions on $N^+_n$.
\begin{defn}\label{YDef}
Given $M \in N^+_n$
and a triple of integers $a,b,c$ such that $1 \leq a < b < c \leq n$,
the corresponding Y-variables are given by ($\Delta_I$ stands for $\Delta_I(M)$)
\begin{align}
\label{YYa}
{}& Y_{abc}(M) = 
\frac{\Delta_{[a,b-1]\cup[c+1,n]} \Delta_{[a+1,b]\cup[c,n]} }
  { \Delta_{[a,b]\cup[c+1,n]} \Delta_{[a+1,b-1]\cup[c,n]} } \,, \quad
\tilde{Y}_{abc}(M) = 
\frac{\Delta_{[1,a]\cup[b+1,c-1]} \Delta_{[1,a-1]\cup[b,c]} }
  { \Delta_{[1,a]\cup[b,c-1]} \Delta_{[1,a-1]\cup[b+1,c]} } \,,  \\[2mm]
\label{YYb}  
{}&   Y^{abc}(M) = 
\frac{ \Delta^{[a,b]\cup[c+1,n]} \Delta^{[a+1,b-1]\cup[c,n]} }
  { \vphantom{\Delta^{A^A}} 
  \Delta^{[a,b-1]\cup[c+1,n]} \Delta^{[a+1,b]\cup[c,n]} } \,, \quad
\tilde{Y}^{abc}(M) = 
\frac{ \Delta^{[1,a]\cup[b,c-1]} \Delta^{[1,a-1]\cup[b+1,c]} }
  { \vphantom{\Delta^{A^A}} 
  \Delta^{[1,a]\cup[b+1,c-1]} \Delta^{[1,a-1]\cup[b,c]}  } \,.  
\end{align}
\end{defn}

Below, $\Delta_I^J (M)$ will denote the minor of $M$ corresponding
to   intersection of the rows labeled by the set $I$
and the columns labeled by the set $J$. 

\begin{lem}\label{YYlem}
The Y-variables  satisfy the following relations:
\begin{align}
 \label{yty1}
{}&  Y_{abc}(M) = \tilde{Y}^{a,a+c-b,c}(M)  = 
 \frac{\Delta_{[a,b-1]}^{[a+c-b+1,c]}(M) \, \Delta_{[a+1,b]}^{[a+c-b,c-1]}(M)}%
  {\Delta_{[a,b]}^{[a+c-b,c]}(M) \, \Delta_{[a+1,b-1]}^{[a+c-b+1,c-1]}(M) } , & \\[1mm]
\label{yty2}
{}    
 \tilde{Y}_{abc}(M') =&\, 
 \frac{1}{Y_{n+1-c,n+1-b,n+1-a} (M)^{\vphantom{d^d}} }  , \qquad
 \tilde{Y}^{abc}(M'') = 
 \frac{1}{Y^{n+1-c,n+1-b,n+1-a} (M)^{\vphantom{d^d}} } . &
\end{align} 
\end{lem}
It is worth mentioning here that, thanks to relation (\ref{invol}),  we can 
exchange $M' \leftrightarrow M$ in the
first relation in (\ref{yty2}) and  $M'' \leftrightarrow M$ in the second. 

Relation (\ref{yty1}) shows that the set of $Y$-variables 
is somewhat redundant. 
Nevertheless, we will keep both $Y_{abc}$ and 
$\tilde{Y}^{abc}$ in order to treat the right and upper minors
on an equal footing.  In particular, we will need the values of the $Y$-variables
in the following case:

\begin{lem}\label{M0}
Let $x>0$. Let $M_x \in N_n^+$ be given by (\ref{Mjac}), where $x_{ij}=x$
for all $i,j$. The corresponding $Y$-variables do not depend on $x$ and are
given by 
\begin{align}\label{Y0}
Y_{abc}(M_x) =  \tilde{Y}_{abc}(M_x) = \frac{c-b}{b-a} , \qquad
Y^{abc}(M_x) =  \tilde{Y}^{abc}(M_x) = \frac{b-a}{c-b} .
\end{align}
\end{lem}

In terms of the $Y$-variables (omitting their 
dependence  on $M$),  formula (\ref{LL1}) reads 
\begin{align} 
\nonumber
 {\mathcal L}(M) = 
l\bigl( Y_{123} \bigr) + l\bigl(Y_{124} \bigr) +
l\bigl(Y_{134}\bigr) +  l\bigl(Y_{234}\bigr) 
=
l\bigl( \tilde{Y}^{123} \bigr) + l\bigl(\tilde{Y}^{134} \bigr) +
l\bigl(\tilde{Y}^{124}\bigr) +  l\bigl(\tilde{Y}^{234}\bigr) .
\end{align} 
A natural generalization of this expression is the sum of dilogarithms whose arguments
are the $Y$-variables labeled by the points  of the discrete tetrahedron $T_n$, $n \geq 3$, 
\begin{align*}
   T_n = \bigl\{ \ \{a,b,c\} \in {\mathbb N}^3 \,\bigm| \,
   1 \leq a < b < c \leq n \bigr\}.
\end{align*}

The following statement is the main result of the present paper:
\begin{thm}\label{Dil12}
For every $M \in N^+_n$, the following chain of
dilogarithm identities holds:
\begin{align}
\label{dil1}
& \sum\limits_{ \{a,b,c\}\in T_n } l\Bigl( Y_{abc} (M') \Bigr) \ =
   \sum\limits_{ \{a,b,c\}\in T_n } l\Bigl( \frac{1}{Y_{abc} (M) } \Bigr) =
    \sum\limits_{ \{a,b,c\}\in T_n } l\Bigl( Y_{abc} (M'') \Bigr)  \\[2.5mm] 
\label{dil2}
=& \sum\limits_{ \{a,b,c\}\in T_n } l\Bigl( \tilde{Y}^{abc} (M') \Bigr) \ =
   \sum\limits_{ \{a,b,c\}\in T_n } l\Bigl( \frac{1}{\tilde{Y}^{abc} (M) } \Bigr) =
    \sum\limits_{ \{a,b,c\}\in T_n } l\Bigl( \tilde{Y}^{abc} (M'') \Bigr)
\\[2.5mm]
\label{dil3}
    =& \sum\limits_{ \{a,b,c\}\in T_n } l\Bigl( \tilde{Y}_{abc} (M') \Bigr) \ =
   \sum\limits_{ \{a,b,c\}\in T_n } l\Bigl( \frac{1}{\tilde{Y}_{abc} (M) } \Bigr) =
    \sum\limits_{ \{a,b,c\}\in T_n } l\Bigl( \tilde{Y}_{abc} (M'') \Bigr) 
\\[2.5mm]
\label{dil4}
=& \sum\limits_{ \{a,b,c\}\in T_n } l\Bigl( {Y}^{abc} (M') \Bigr) \ =
   \sum\limits_{ \{a,b,c\}\in T_n } l\Bigl( \frac{1}{{Y}^{abc} (M) } \Bigr) =
    \sum\limits_{ \{a,b,c\}\in T_n } l\Bigl( {Y}^{abc} (M'') \Bigr).
\end{align} 
\end{thm}

\begin{proof}
The key technical step is to prove the following:
\begin{prop}\label{ZZ1}
For every $M \in N^+_n$, the following dilogarithm identities hold:
\begin{align} 
\label{dilsum1}
{}& \sum\limits_{ \{a,b,c\}\in T_n } l\Bigl( Y_{abc} (M) \Bigr) +
   \sum\limits_{ \{a,b,c\}\in T_n } l\Bigl( Y_{abc} (M') \Bigr) 
   = \frac{n(n-1)(n-2)}{6}  , \\
\label{dilsum2}   
{}& \sum\limits_{ \{a,b,c\}\in T_n } l\Bigl( Y^{abc} (M) \Bigr) +
   \sum\limits_{ \{a,b,c\}\in T_n } l\Bigl( Y^{abc} (M'') \Bigr) 
   = \frac{n(n-1)(n-2)}{6} .
\end{align}    
\end{prop}
Along with relation (\ref{rd1}),  Proposition \ref{ZZ1} 
yields the first equality in (\ref{dil1}) and the 
last equality in (\ref{dil4}). Furthermore,  
along with Lemma~\ref{YYlem},
Proposition \ref{ZZ1}  implies the following:
\begin{prop}\label{ZZ2}
For every $M \in N^+_n$, the following dilogarithm identities hold:
\begin{equation}\label{yaux}
\begin{aligned} 
  \sum\limits_{ \{a,b,c\}\in T_n } l\bigl( Y_{abc}(M) \bigr)
 &= \sum\limits_{ \{a,b,c\}\in T_n } l\bigl( \tilde{Y}^{abc}(M) \bigr) \\
&= \sum\limits_{ \{a,b,c\}\in T_n } l\bigl( \tilde{Y}_{abc}(M) \bigr)
= \sum\limits_{ \{a,b,c\}\in T_n } l\bigl( Y^{abc}(M) \bigr) .
\end{aligned} 
\end{equation}
\end{prop}
Clearly, $M$ in (\ref{yaux}) can be replaced by $M'$ or $M''$.
Therefore, we have ``vertical'' equalities of all the sums in the 
left, middle, and right columns in (\ref{dil1})--(\ref{dil4}).
This completes the proof of Theorem~~\ref{Dil12} since,
by Proposition \ref{ZZ1}, we have already established the first 
equality in (\ref{dil1}) and the last equality in (\ref{dil4}).
\end{proof}

\subsection{$S_3$-symmetry of the dilogarithm identities}

Given $M \in N^+_n$ and $s \in S_3$, let $s(M)$ be defined as in (\ref{s3nt}).
Theorem~\ref{Dil12} can be formulated in the following way:
\begin{thm}\label{Dil1234}
For every $M \in N^+_n$ and any $s_1,s_2,s_3,s_4 \in S_3$, 
the following dilogarithm identities hold:
\begin{equation}\label{ysig}
\begin{aligned} 
{}&    \sum\limits_{ \{a,b,c\}\in T_n } 
 l\Bigl( \bigl( Y_{abc}(s_1(M)) \bigr)^{\mathrm{sgn}\, (s_1)} \Bigr)
 =   \sum\limits_{ \{a,b,c\}\in T_n }
 l\Bigl( \bigl( \tilde{Y}^{abc}(s_2(M)) \bigr)^{\mathrm{sgn}\, (s_2)} \Bigr) \\
{}&=  \sum\limits_{ \{a,b,c\}\in T_n } 
 l\Bigl( \bigl( \tilde{Y}_{abc}(s_3(M)) \bigr)^{\mathrm{sgn}\, (s_3)} \Bigr)
 =   \sum\limits_{ \{a,b,c\}\in T_n }
 l\Bigl( \bigl( Y^{abc}(s_4(M)) \bigr)^{\mathrm{sgn}\, (s_4)} \Bigr)  .
\end{aligned} 
\end{equation}
\end{thm}
\begin{proof}
Indeed, for $s_i=id,\sigma_1,\sigma_2$, identities (\ref{ysig}) are those
given in (\ref{dil1})-(\ref{dil4}). For $s_i = \sigma_1   \sigma_2$,  
$\sigma_2   \sigma_1$, identities (\ref{ysig}) are obtained from 
(\ref{dil1})-(\ref{dil4}) by the substitution $M \to M'$ and $M \to M''$,
respectively. And for $s_i=(\sigma_1   \sigma_2   \sigma_1)$, 
$(\sigma_2  \sigma_1   \sigma_2)$, identities (\ref{ysig}) are obtained from 
(\ref{dil1})-(\ref{dil4}) by the substitution $M \to (M')''$ and $M \to (M'')'$.
\end{proof}

In conclusion, we give a generalization of Proposition~\ref{DilG}:
\begin{thm}\label{DilGG}
For every $G \in B^+_n$,  identities (\ref{dil1})--(\ref{dil4}) remain
valid if $M$, $M'$, and $M''$ are replaced by $G$, $\check{G}$, 
and $\hat{G}$, respectively.
\end{thm}
\noindent
(The proof is given in the Appendix.)

\appendix
\section{Appendix}

\begin{proof}[\bf Proof of Lemma~\ref{GM}] 
If $M \in N^+_n$, then all the leading principal 
minors of the matrix $M P$ are non zero.
Therefore,  the  matrix $M P$ admits an LU decomposition. 
Taking $L$ unipotent and 
writing $L = P M' P$ and $U= D_{\!\scriptscriptstyle M} M''$,
where $M'$ and $M''$ are unipotent, we obtain the decomposition~(\ref{gauss}).
Its uniqueness is due to Corollary 3.5.6 in~\cite{HJ}.

Now, we have to check that $M', M'' \in N^+_n$.
If $M$ is given by (\ref{Mjac}), then by Lemma~2.4.7  in \cite{BFZ}, we have
\begin{align}\label{mab}
  \Delta_{[a,b]}(M) = \prod_{i=1}^{b-a+1} \prod_{j=b+1}^n  x_{ij}  .
\end{align}
By reversing this formula, we obtain
\begin{align}\label{jcmm}
    x_{ij}(M) = \frac{\Delta_{[j-i,j-1]}(M)\,\Delta_{[j-i+2,j]}(M)}%
{\Delta_{[j-i+1,j-1]}(M)\,\Delta_{[j-i+1,j]}(M)}  \,.
\end{align}

Let us rewrite equation (\ref{gauss}) as 
$P M =M' (P D_{\!\scriptscriptstyle M} P) (P M'' P)$ and 
$MP =(PM'P)D_{\!\scriptscriptstyle M} M''$.
Then,  the Binet-Cauchy formula yields
\begin{align}\label{delij} 
 \varepsilon_I \, \Delta_{\bar{I}}(M) =
\Delta^{[1,|I|]}(D_{\!\scriptscriptstyle M}) \, \Delta_{I}(M') , \qquad 
 \varepsilon_J \, \Delta^{\bar{J}}(M) =
\Delta^{[1,|J|]}(D_{\!\scriptscriptstyle M}) \, \Delta^{J}(M'') ,
\end{align}
where  $\varepsilon_I = (-1)^{|I|(|I|-1)/2}$.
For $J=[1,i]$, the second equation in (\ref{delij}) yields  
\begin{align}\label{vre}
   (-1)^{i(i-1)/2} \Delta_{[1,i]}(M) =
  \Delta^{[1,i]}(D_{\!\scriptscriptstyle M})=
(D_{\!\scriptscriptstyle M})_{11}\ldots (D_{\!\scriptscriptstyle M})_{ii} .
\end{align}

{} From (\ref{vre})  and (\ref{mab}), we find $D_{\!\scriptscriptstyle M}$:
\begin{align}\label{dmk} 
  (D_{\!\scriptscriptstyle M})_{ii} = 
  (-1)^{i-1} \frac{ \prod_{k=i+1}^n x_{ik} }{ \prod_{k=1}^{i-1} x_{ki} } .
\end{align}
Note that the sign of $\Delta^{[1,|I|]}(D_{\!\scriptscriptstyle M})$ is equal to 
$\varepsilon_I$. Therefore, equations (\ref{delij})   imply that 
 the flag minors $\Delta_{I}(M')$ and $\Delta^{J}(M'')$ are positive 
for all $I,J$. Hence $M' \in N_n^+$ thanks to equation (\ref{jcmm}) and 
Lemma~\ref{Jlem}. Similarly, the positivity of all flag minors $\Delta^{J}(M'')$ 
implies that $M'' \in N_n^+$ by Theorem~1.8 in~\cite{BFZ}.
\end{proof}

\begin{proof}[\bf Proof of Lemma~\ref{BBc} and Lemma~\ref{GMB}] 
Equations (\ref{delij}) imply that 
\begin{align}\label{djim}
\Delta^{[1,|I|]}(\tilde{D}_{\!\scriptscriptstyle M}) \, \Delta_{I}(M')
=   \Delta_{\bar{I}} (M) , \qquad
\Delta^{[1,|J|]}(\tilde{D}_{\!\scriptscriptstyle M}) \, \Delta^{J}(M'') 
=    \Delta^{\bar{J}} (M) .
\end{align}
Using the first of these relations, we get
\begin{align*} 
{}   \Delta_{\bar{I}} (G) &= \Delta_{\bar{I}} (M \, \Lambda)
= \Delta_{\bar{I}} (M) \, \Delta_{[n+1-| I|,n]} (\Lambda) \\
{}&  
= \Delta_{I}(M') \, \Delta_{[n+1-| I|,n]} (\Lambda \, P  \tilde{D}_{\!\scriptscriptstyle M}  P)
= \Delta_{I}(M' \Lambda \, P  \tilde{D}_{\!\scriptscriptstyle M}  P) .
\end{align*}
Thus,   $\check{G}$ given by (\ref{gc}) is a solution
to the first equation in (\ref{ggc}). Using the second relation in~(\ref{djim}), 
one can check in the same way that $\hat{G}$ given by (\ref{gh})
is a solution to the second equation in (\ref{ggc}). 
Let us prove that these solutions are unique. Given $G \in B^+_n$, suppose that
$G_1, G_2 \in B^+_n$ are such that 
\begin{align}\label{ggg}
\Delta_{I} (G_1) = \Delta_{\bar{I}} (G) = \Delta_{I} (G_2)
\end{align}
for all $I \subset [1,n]$. $G_1$ and $G_2$ are uniquely represented as 
$G_l = M_l \, \Lambda_l$, $l=1,2$, where $M_l \in N^+_n$ and $\Lambda_l$ 
are diagonal matrices with positive diagonal entries. 
Taking $I=[k,n]$ in (\ref{ggg}), we see that 
$\prod_{i=k}^n (\Lambda_1)_{ii} = \prod_{i=k}^n (\Lambda_2)_{ii}$
for all $k=1,\ldots,n$. Hence $\Lambda_1 = \Lambda_2$. 
 
Since $\Delta_{I} (G_l)  = \Delta_{I} (M_l) \, \Delta_{[n+1-| I|,n]} (\Lambda_l)$
and we know that $\Lambda_1 = \Lambda_2$,
we infer from (\ref{ggg}) that $\Delta_{I} (M_1) = \Delta_{I} (M_2)$
for all $I$. Whence, taking $I=[i,j]$ and using (\ref{jcmm}),
we conclude that the Jacobi coordinates of $M_1$ and $M_2$
coincide. Hence $M_1 = M_2$ and so $G_1=G_2$. 
The proof of the uniqueness of $\hat{G}$ is analogous.  
\end{proof}

\begin{proof}[\bf Proof of Proposition~\ref{SG}] 

a) 
Equation (\ref{gauss}) can be rewritten as follows:
\begin{align}\label{ga1}
 P M' P = M P  (D_{\!\scriptscriptstyle M})^{-1} 
 \bigl(  D_{\!\scriptscriptstyle M} (M'')^{-1} D_{\!\scriptscriptstyle M}^{-1} \bigr)  .
\end{align}
Comparing (\ref{ga1}) with (\ref{gauss}), we see that
$(M')' =M$, $D_{\!\scriptscriptstyle M'} = D_{\!\scriptscriptstyle M}^{-1}$,
and $(M')'' = D_{\!\scriptscriptstyle M} (M'')^{-1} D_{\!\scriptscriptstyle M}^{-1}$.
Similarly, one finds that 
$(M'')'' =M$, $D_{\!\scriptscriptstyle M''} = D_{\!\scriptscriptstyle M}^{-1}$, and
$(M'')' = P  D_{\!\scriptscriptstyle M}^{-1}  P (M')^{-1} P D_{\!\scriptscriptstyle M} P$.

b)
Let $u,v,z \in N^+_n$ satisfy the following relation
\begin{align}\label{uvw}
  P z^t = v^ t\, d \, u ,
\end{align}
where $d$ is a diagonal matrix and $t$ denotes the matrix transposition operation.
Theorem 1.4 in \cite{BFZ} relates the Jacobi coordinates of $u$ (for a
factorization corresponding to an arbitrary minimal reduced word for $w_0$)
to the flag minors of $z$.  In particular, for the order of the factors as in (\ref{Mjac}),
this theorem yields
\begin{align}\label{tdtd}
 & \frac{\Delta^{[i,j]}(z) \, \Delta^{[i+1,j-1]}(z) }{\Delta^{[i,j-1]}(z) \,\Delta^{[i+1,j]}(z)} 
 = x_{n+1-j,n+1-i} (u)  .
\end{align}
Note that (\ref{uvw}) matches (\ref{gauss}) if we substitute $u$ by $M''$ and
$z$ by $PM^tP$. Therefore, using (\ref{tdtd}) and  taking into account that 
$\Delta^{[i,j]}(PM^tP) = \Delta_{[n+1-j,n+1-i]}(M)$, we obtain  
\begin{align}\label{xmmm}
  x_{ij}(M'') = \frac{\Delta_{[i,j]}(M)\,\Delta_{[i+1,j-1]}(M)}%
{\Delta_{[i,j-1]}(M)\,\Delta_{[i+1,j]}(M)}   \,.
\end{align}
With the help of formula (\ref{mab}), we obtain from (\ref{xmmm}) the 
second equation in (\ref{xxm}).

Let $S \in B_n$ denote the diagonal matrix such that $S_{ii}=(-1)^{i-1}$.
Equation (2.11.3) in~\cite{BFZ}, along with the relation $S J_k(x) S = J_k(-x)$, 
implies that if $M \in N_n^+$, then $S M^{-1} S \in N_n^+$ and 
 $x_{ij}(S M^{-1} S) =  x_{n+1-j,n+1-i}(M)$. Note also that $S P S = (-1)^{n-1} P$.
Therefore, multiplying (\ref{gauss}) by $S$ on both sides and taking the matrix inverse,
we get a counterpart of (\ref{uvw}), where $u$ is substituted by $S(M')^{-1}S$ 
and $z$ by $P S (M^{-1})^t S P$. 
Thus, we can replace $M''$ by $S (M')^{-1} S $ and $M$ by $S M^{-1} S $
in (\ref{xmmm}) and hence also in the second equation in (\ref{xxm}).
By doing so,  we obtain the first equation in (\ref{xxm}):
\begin{align*}
{}    x_{ij}(M')  =  x_{n+1-j,n+1-i} \bigl( S (M')^{-1} S \bigr)   
       = \frac{ \prod\limits_{k=n+2-i}^{n} x_{j+1-i,k}( S M^{-1} S) }%
{ \prod\limits_{k=n+1-i}^{n} x_{j-i,k} ( S M^{-1} S)}  =
  \frac{ \prod\limits_{k=1}^{i-1} x_{k,n+i-j}(M) }%
{ \prod\limits_{k=1}^{i} x_{k,n+1+i-j} (M)} .
\end{align*}
\end{proof}

\begin{proof}[\bf Proof of Proposition~\ref{MS3}] 

a) Given $M \in N_n^+$, let $Q_i(M)$ denote the ratio of the l.h.s. to the 
r.h.s. of (\ref{con}). A straightforward computation using (\ref{xxm})
shows that $Q_i(M')=1/Q_i(M)$ and $Q_i(M'')=1/Q_i(M)$.
Since $M \in \tilde{N}^+_n$ iff $Q_i(M)=1$ for all $i<n/2$, the claim follows.  

b) The claim can be proved by computing the Jacobi coordinates of both 
 sides of (\ref{Mrho}) with the help of (\ref{xxm}). However, this computation 
 is rather tedious, so   we will give another proof. 
 First, we infer from  (\ref{ga1}) that  
\begin{align} 
\nonumber
 P (M')'' P &= P \bigl(  D_{\!\scriptscriptstyle M} (M'')^{-1} D_{\!\scriptscriptstyle M}^{-1} \bigr) P=
 P D_{\!\scriptscriptstyle M} P M^{-1}  P M'  \\
\label{ga2} 
 &= \bigl( P D_{\!\scriptscriptstyle M} P M^{-1}  P D_{\!\scriptscriptstyle M}^{-1} P \bigr) 
  P D_{\!\scriptscriptstyle M} M' .
\end{align}
Comparing (\ref{ga2}) with the r.h.s. of (\ref{gauss}), we see that 
$((M')'')'=P D_{\!\scriptscriptstyle M} P M^{-1}  P D_{\!\scriptscriptstyle M}^{-1} P$.
In a similar way, one finds that  $((M'')')''=
 D_{\!\scriptscriptstyle M}^{-1}  M^{-1} D_{\!\scriptscriptstyle M}$.
Thus $((M')'')'= ((M'')')''$ holds iff $M$ commutes 
with $(D_{\!\scriptscriptstyle M} P)^2$. But the latter is a diagonal
matrix and, since all the entries of $M$ above the diagonal are positive,
we conclude that $(D_{\!\scriptscriptstyle M} P)^2$ must be a multiple
of the unit matrix. Furthermore, we have
$\det\bigl( (D_{\!\scriptscriptstyle M} P)^2\bigr) =
\det\bigl( D_{\!\scriptscriptstyle M}^2 \bigr)= 
\det\bigl( M^2 \bigr)=1$. 
Taking formula (\ref{dmk}) into account, we conclude that 
$((M')'')'= ((M'')')''$ holds iff 
$(D_{\!\scriptscriptstyle M})_{ii} (D_{\!\scriptscriptstyle M})_{n+1-i,n+1-i} = (-1)^{n+1}$
for all~$i$. The latter condition is equivalent to (\ref{con}).
Taking into account that $\det{D_{\!\scriptscriptstyle M}} = \pm 1$, 
we infer that conditions for $i \geq n/2$ are equivalent to these for $i< n/2$.
\end{proof}

\begin{proof}[\bf Proof of  Proposition~\ref{SY}] 
Relations  (\ref{my1}) follow directly from~(\ref{xxm}). Using them, we obtain
\begin{align*} 
{}& y_{ij}\bigl( ((M')'')' \bigr) = 
 \bar{y}_{i,n+i-j}\bigl( (M')'' \bigr) =
 y_{n-j,n+i-j}\bigl( M' \bigr) =
 \bar{y}_{n-j,n-i} (M) , \\[0.5mm] 
{}& y_{ij}\bigl( ((M'')')'' \bigr) = 
 \bar{y}_{j-i,j}\bigl( (M'')' \bigr) =
 y_{j-i,n-i}\bigl( M'' \bigr) =
 \bar{y}_{n-j,n-i} (M) .
\end{align*}
Whence follows the first equality in (\ref{my2}). The second equality in (\ref{my2})
follows from the first one by noting that $\bar{y}_{ij}(M) = 1/y_{ij}(M)$.
\end{proof}

\begin{proof}[\bf Proof of Lemma~\ref{YYlem}] 
Recall that $M_{ij}=0$ if $j<i$. Therefore,
$\Delta_{[a,b]\cup[c,n]}= 
\Delta_{[a,b]\cup[c,n]}^{[a+c-b-1,n]}
=\Delta_{[a,b]}^{[a+c-b-1,c-1]} \Delta_{[c,n]}^{[c,n]}
=\Delta_{[a,b]}^{[a+c-b-1,c-1]} = 
\Delta_{[1,a-1]}^{[1,a-1]} \Delta_{[a,b]}^{[a+c-b-1,c-1]}
= \Delta_{ [1,b]}^{[1,a-1] \cup [a+c-b-1,c-1]}=
\Delta^{[1,a-1] \cup [a+c-b-1,c-1]}$.
Thus, 
\begin{align} \label{upr}
 \Delta_{[a,b]\cup[c,n]}= 
 \Delta_{[a,b]}^{[a+c-b-1,c-1]}
 =\Delta^{[1,a-1] \cup [a+c-b-1,c-1]}  \,.
\end{align}
Relations (\ref{yty1}) are easily checked with the help of (\ref{upr}). 

Set $I_1=[1,a]\cup[b+1,c-1]$, $I_2 =[1,a-1]\cup[b+1,c]$, 
$I_3 =[1,a-1]\cup[b,c]$, $I_4= [1,a]\cup[b,c-1]$.
Note that $|I_1|=|I_2|$ and $|I_3|=|I_4|$. Combining the second 
equation in (\ref{YYa}) with
the first relation  in (\ref{delij}), we obtain   
\begin{align} \label{uu}
  \tilde{Y}_{abc}(M') = 
  \frac{ \Delta_{I_1} (M')  \Delta_{I_3} (M')}{ \Delta_{I_2} (M')  \Delta_{I_4} (M')}
  = \frac{ \Delta_{\bar{I}_1} (M)  \Delta_{\bar{I}_3} (M)}
  	{ \Delta_{\bar{I}_2} (M)  \Delta_{\bar{I}_4} (M)} .
\end{align}
Comparing the r.h.s. of (\ref{uu}) with the first equation in (\ref{YYa}), 
we obtain the first relation in (\ref{yty2}). The second relation in (\ref{yty2})
is derived analogously by combining the second equation in (\ref{YYb}) with 
the second relation in (\ref{delij}).
\end{proof}

For the proof of  Lemma~\ref{M0} we will need the following statement.

\begin{lem}\label{Toep}
Let\, $T_{k,m}$, $k,m \in {\mathbb N}$, be the $(m\,{+}\,1)\,{\times}\, (m\,{+}\,1)$
Toeplitz matrix such that
\begin{align}\label{Tij}
 \bigl( T_{k,m} \bigr)_{ij} = \frac{1}{(k+j-i)!} \,, \qquad 1 \leq i,j \leq m+1 ,
\end{align}  
where $1/p! = 0$ if $p <0$.
The determinant of $T_{k,m}$ is given by
\begin{align}\label{Fkm}
  F_{k,m} \equiv \det \bigl( T_{k,m} \bigr) = \frac{1!\, 2! \ldots m!}{k! (k+1)! \ldots (k+m)!} .
\end{align} 
\end{lem}
  
\begin{proof}  Equation (\ref{Fkm}) is easily checked for $m=0,1$.
In order to prove it by induction for $m>1$, we use the well-known identity 
(that can be checked with the help of  (\ref{PluJ}) and (\ref{upr}))
\begin{align}\label{m4}
 \det M \, \det M_0 + \det M_{12} \, \det M_{21} = \det M_{11} \, \det M_{22} ,
\end{align}
where $M$ is an arbitrary $(m\,{+}\,1)\,{\times}\, (m\,{+}\,1)$ matrix, 
$M_{ij}$ are its  corner $m\,{\times}\, m$ submatrices, and 
$M_0$ is the central $(m\,{-}\,1)\,{\times}\, (m\,{-}\,1)$ submatrix. 
 For $T_{k,m}$,  identity (\ref{m4}) yields
\begin{align}\label{FF1}
  F_{k,m} = \bigl( F_{k,m-1} F_{k,m-1} - F_{k+1,m-1} F_{k-1,m-1}  \bigr)/ F_{k,m-2} .
\end{align}
It is straightforward to check that  substitution into (\ref{FF1}) of the expressions for 
$F_{k,m-1}$,  $F_{k+1,m-1}$, and $F_{k,m-2}$ given by (\ref{Fkm})
yields the desired expression for $F_{k,m}$.
\end{proof}

\begin{proof}[\bf Proof of Lemma~\ref{M0}]
First, we note that if $M^{(n)}_x \in N_n^+$ is given by (\ref{Mjac}), where $x_{ij}=x$
for all $i,j$, then the entries of $M^{(n)}_x$ are expressed in terms of 
binomial coefficients:
\begin{align}\label{M0ij}
    \bigl(M^{(n)}_x \bigr)_{ij} = {{n -i} \choose {j-i} } \, x^{j-i} .
\end{align} 
Indeed, (\ref{M0ij}) is obvious for $n=2$. Note that in (\ref{Mjac}) we have
$( {\mathbb J}_{n-1} )_{ij} = \delta_{ij} + x \delta_{i,j-1}$. 
Therefore, equation  (\ref{M0ij}) for $n>2$ is derived by induction:
\begin{align}\nonumber
    \bigl(M^{(n)}_x \bigr)_{ij} &=  
    \bigl( (M^{(n-1)}_x \oplus 1) \, {\mathbb J}_{n-1} \bigr)_{ij} 
     = \bigl(M^{(n-1)}_x \bigr)_{ij} + x  \bigl(M^{(n-1)}_x \bigr)_{i,j-1} \\
\nonumber {}&
     = x^{j-i} \, {{n-1-i} \choose {j-i} }  +  x^{j-i} \, {{n-1-i} \choose {j-i-1} }
 = x^{j-i} \, {{n-i} \choose {j-i} }  .
\end{align}  
Equation (\ref{M0ij}) implies that, in the row $i$, all the entries have a common
factor $x^{-i} (n-i)!$ and, in the column $j$, all the entries have a common factor 
$x^{j}/ (n-j)!$. Hence, using formula (\ref{upr}), we obtain 
\begin{align} \nonumber
{} \Delta_{[a,b]\cup[c,n]} \bigl(M^{(n)}_x \bigr) &= 
 \Delta_{[a,b]}^{[a+c-b-1,c-1]}\bigl(M^{(n)}_x \bigr) \\
\nonumber
{}&   = \Bigl( \prod_{i=a}^b \frac{ (n-i)! }{x^i} \Bigr)
  \Bigl( \prod_{j=a+c-b-1}^{c-1} \frac{x^j}{ (n-j)! } \Bigr) 
  \,   \det (T_{c-b-1,b-a}) ,
\end{align}
where $T_{k,m}$ is the Toeplitz matrix (\ref{Tij}). Then, by Lemma~\ref{Toep},
we have
\begin{align}\label{Dmx}
{}& \frac{ \Delta_{[a,b]\cup[c,n]} \bigl(M^{(n)}_x \bigr) }%
 { x^{(c-b-1)(b+1-a)} }   
  = \frac{ \prod_{i=a}^b   (n-i)!}{ \prod_{j=a+c-b-1}^{c-1}   (n-j)!} \,
  \frac{1!\, 2! \ldots (b-a)!}{(c-b-1)! (c-b)! \ldots (c-a-1)!} .
\end{align}
Combining this formula with the first equation in (\ref{YYa}), we infer that
$Y_{abc}(M_x)= \frac{c-b}{b-a}$. 
Whence, by (\ref{yty1}), we have  
$\tilde{Y}^{abc}(M_x)= Y_{a,a+c-b,c}(M_x)=\frac{b-a}{c-b}$. 
 
Relations (\ref{xxm}) imply that $M'_x=M''_x=M_{x^{-1}}$.
Therefore
$Y_{abc}(M'_x) = Y_{abc}(M''_x) = Y_{abc}(M_{x^{-1}})
= Y_{abc}(M_x)$ since the latter does not depend on $x$.
Combining these relations with equations (\ref{yty2}), we obtain 
$\tilde{Y}_{abc}(M_x) = 
 1/{Y_{n+1-c,n+1-b,n+1-a} (M_x)  }=\frac{c-b}{b-a}$ and
$Y^{abc}(M_x) = 
 1/{\tilde{Y}^{n+1-c,n+1-b,n+1-a} (M_x) } =  \frac{b-a}{c-b}$.
\end{proof}

Let ${\mathbb R}(t)^*$ be the multiplicative group of nonzero 
rational functions in one or several variables,
$t=(t_1,t_2,\ldots)$. Let $ {\mathbb R}(t)^*  \wedge  {\mathbb R}(t)^* $ 
be the abelian group generated by formal elements
$x \wedge y$, where $x, y \in {\mathbb R}(t)^*$, subject to the following relations:
\begin{align}\label{wed1}
   x \wedge y = - y \wedge x , \qquad
   (xy) \wedge z = x \wedge z + y \wedge z .
\end{align} 
These relations imply, in particular, that
\begin{align}\label{wed2}
   x \wedge 1 = 0 , \qquad
  \frac{1}{x} \wedge y = - x \wedge y.
\end{align} 

The proof of  Proposition~\ref{ZZ1} will be based on the following 
statement from \cite{Zag1} (see also Chapter II, Section 2A in \cite{Zag2}):
\begin{lem}[\cite{Zag1}, \S\,7, Prop.\,1]\label{ZD}
Let $X_1,\ldots,X_m \in {\mathbb R}(t)$ be a set of functions such that
\begin{align}\label{sumx}
   \sum_{k=1}^m X_k \wedge (1 - X_k) =0  
\end{align} 
in ${\mathbb R}(t)^*  \wedge  {\mathbb R}(t)^*$.
Then $\sum_{k=1}^m L(X_k)$ does not depend on $t$.
\end{lem}
\begin{proof}[\bf Proof of  Proposition~\ref{ZZ1}] 
Let $I \subset [1,n]$, $J \subset [1,n]$, and 
let the triple $1\leq a < b < c \leq n$ be such that
$I \cap \{a,b,c\} = J \cap \{a,b,c\} = \emptyset$.
Then the right and upper flag minors of any square
matrix $M$ satisfy the following Pl\"ucker relations 
(see, e.g., Proposition 2.6.2. in \cite{BFZ}): 
\begin{align}
\label{PluJ}
{} \Delta_{I\cup\{a,c\}} \Delta_{I\cup\{b\}} &= 
\Delta_{I\cup\{a,b\}} \Delta_{I\cup\{c\}} +
\Delta_{I\cup\{b,c\}} \Delta_{I\cup\{a\}} , \\[0.5mm]
\label{PluI}
{} \Delta^{J\cup\{a,c\}} \Delta^{J\cup\{b\}} &= 
\Delta^{J\cup\{a,b\}} \Delta^{J\cup\{c\}} +
\Delta^{J\cup\{b,c\}} \Delta^{J\cup\{a\}} .
\end{align} 
Let $M \in N_n^+$ be given by (\ref{Mjac}). Consider
functions $X_{abc}, X'_{abc} \in {\mathbb R}(x_{12},x_{13},\ldots)^*$
given by
\begin{align}\label{xy} 
 X_{abc} = \frac{Y_{abc}(M)}{1+Y_{abc}(M)} , \qquad
 X'_{abc} = \frac{Y_{abc}(M')}{1+Y_{abc}(M')} .
\end{align}  
Using Pl\"ucker relations (\ref{PluJ}) and 
the first relation in (\ref{yty2}), we obtain 
(in this proof, all the minors $\Delta_I$ are those of $M$)
\begin{align}
\label{xx1} 
{}& X_{abc} = 
\frac{\Delta_{[a,b-1]\cup[c+1,n]} \Delta_{[a+1,b]\cup[c,n]} }
  { \Delta_{[a+1,b]\cup[c+1,n]} \Delta_{[a,b-1]\cup[c,n]} } , \qquad 
  1- X_{abc} =  
\frac{\Delta_{[a,b]\cup[c+1,n]} \Delta_{[a+1,b-1]\cup[c,n]} }
  { \Delta_{[a+1,b]\cup[c+1,n]} \Delta_{[a,b-1]\cup[c,n]} } , \\ 
\label{xxx1}   
{}& X'_{abc} =  
\frac{\Delta_{[1,n-c]\cup[n+2-b,n+1-a]} 
 \Delta_{[1,n+1-c]\cup[n+1-b,n-a]} }
  { \Delta_{[1,n-c]\cup[n+1-b,n-a]} 
 \Delta_{[1,n+1-c]\cup[n+2-b,n+1-a]} }  , \\
\label{xxx2} 
{}& 1 -X'_{abc} =
\frac{\Delta_{[1,n-c]\cup[n+1-b,n+1-a]} 
 \Delta_{[1,n+1-c]\cup[n+2-b,n-a]} }
  { \Delta_{[1,n-c]\cup[n+1-b,n-a]} 
 \Delta_{[1,n+1-c]\cup[n+2-b,n+1-a]} } .
\end{align}  
We claim that 
\begin{align}\label{sumX}
 \sum_{\{a,b,c\}\in T_n } X_{abc} \wedge (1 - X_{abc}) + 
 \sum_{\{a,b,c\}\in T_n } X'_{abc} \wedge (1 - X'_{abc})
   =0 .
\end{align} 
Indeed, consider the terms in (\ref{sumX}) which
contain $\Delta_{[\alpha,\beta]\cup[\gamma,n]}$.
In the generic case, $\alpha>1$, $\gamma-\beta >1$,
the minor $\Delta_{[\alpha,\beta]\cup[\gamma,n]}$
is not present in $X'_{abc}$ and $(1 - X'_{abc})$,
whereas $X_{abc}$ contains it if $a=\alpha$, $b=\beta+1$,
and $c=\gamma-1$ or $c=\gamma$ and also if
$a=\alpha-1$, $b=\beta$, and $c=\gamma-1$ or $c=\gamma$. 
Similarly, $(1 - X_{abc})$ contains this minor if 
$b=\beta$, $c = \gamma-1$, and $a=\alpha-1$ or $a=\alpha$
and also if $b=\beta+1$, $c = \gamma$, and 
$a=\alpha-1$ or $a=\alpha$. Using relations (\ref{wed1})
and (\ref{wed2}), we can write the contribution of all
these terms to (\ref{sumX}) as 
\begin{align}\label{con1}
\Delta_{[\alpha,\beta]\cup[\gamma,n]} \wedge 
\frac{(1 - X_{\alpha,\beta+1,\gamma-1})}{(1-X_{\alpha,\beta+1,\gamma})}
\frac{(1 - X_{\alpha-1,\beta,\gamma})}{(1-X_{\alpha-1,\beta,\gamma-1})} 
\frac{ X_{\alpha-1,\beta,\gamma-1}}{ X_{\alpha,\beta,\gamma-1}}
\frac{  X_{\alpha,\beta+1,\gamma}}{ X_{\alpha-1,\beta+1,\gamma}}.
\end{align} 
Substituting here the expressions (\ref{xx1}) for
$X_{abc}$ and $(1-X_{abc})$, it is straightforward to check that
(\ref{con1}) is equal to 
$\Delta_{[\alpha,\beta]\cup[\gamma,n]} \wedge 1$ which vanishes
by~(\ref{wed2}). 

The minor $\Delta_{[1,\alpha]\cup[\beta,\gamma]}$
is not present in $X_{abc}$ and $(1 - X_{abc})$ 
in the generic case, $\gamma <n$, $\beta-\alpha>1$, 
and an absolutely analogous consideration shows that
its contribution to (\ref{sumX}) also vanishes.

There are two special cases when a minor is present 
in both sums in (\ref{sumX}). First, it is the case of
$\Delta_{[1,\beta]\cup[\gamma,n]}$, $\gamma-\beta>1$.
Using relations (\ref{wed1}) and (\ref{wed2}),  we 
can write the contribution to (\ref{sumX}) of the terms 
which contain this minor as 
\begin{align}\label{con2}
\Delta_{[1,\beta]\cup[\gamma,n]} \wedge 
\frac{(1 - X_{1,\beta+1,\gamma-1})}{(1-X_{1,\beta+1,\gamma})}
\frac{ X_{1,\beta+1,\gamma}}{ X_{1,\beta,\gamma-1}}
\frac{(1 - X'_{1,n+2-\gamma,n-\beta})}{(1-X'_{1,n+2-\gamma,n+1-\beta})} 
\frac{  X'_{1,n+2-\gamma,n+1-\beta}}{ X'_{1,n+1-\gamma,n-\beta}}.
\end{align} 
Again, it is straightforward to check with the help of (\ref{xx1}), (\ref{xxx1}), 
and (\ref{xxx2})
  that (\ref{con2}) is equal to 
$\Delta_{[1,\beta]\cup[\gamma,n]} \wedge 1$ and so it vanishes.

In the second case, $\Delta_{[\alpha,n]}$, $\alpha>1$,
(or $\Delta_{[1,\alpha]}$, $\alpha<n$, which can be treated 
in the same way)
the situation is somewhat different. This minor is present in 
all $X_{abc}$ such that $a=\alpha-1$, $c=b+1$ and 
in all $(1-X_{abc})$ such that $b=a+1$, $c=\alpha$.
It is also present
in $X'_{1,n+2-\alpha,n}$ and $(1-X'_{1,n+1-\alpha,n})$.
The contribution of these terms to (\ref{sumX}) is
given by 
\begin{align}\label{con3}
\Delta_{[\alpha,n]} \wedge 
\frac{(1-X'_{1,n+2-\alpha,n})}{ X'_{1,n+1-\alpha,n} } \,
\prod_{b = \alpha}^{n-1} (1-X_{\alpha-1,b,b+1}) \,
\prod_{a = 1}^{\alpha-2} \frac{1}{ X_{a,a+1,\alpha}} .
\end{align} 
One can observe that both products in (\ref{con3})
are ``telescopic", so that (\ref{con3}) is equal to
\begin{align}\label{con3b}
\Delta_{[\alpha,n]} \wedge 
\frac{(1-X'_{1,n+2-\alpha,n})}{ X'_{1,n+1-\alpha,n} } \,
 \frac{\Delta_{[\alpha-1,n-1]} \, \Delta_{[\alpha+1,n]}}
{\Delta_{\{\alpha-1\}\cup[\alpha+1,n]} \, \Delta_{[\alpha,n-1]}}
\frac{\Delta_{\{\alpha-1\}\cup[\alpha+1,n]} 
 \, \Delta_{\{1\}\cup [\alpha,n]}}
{\Delta_{\{1\}\cup [\alpha+1,n]} \, \Delta_{[\alpha-1,n]}}  .
\end{align}  
Using (\ref{xxx1}) and (\ref{xxx2}), we see that (\ref{con3b}) 
is equal to 
$\Delta_{[\alpha,n]} \wedge 1$ and 
so it also vanishes.

Thus, we have checked that (\ref{sumX}) holds. Whence,
by Lemma~\ref{ZD}, 
\begin{align}\label{sumXX}
\sum_{\{a,b,c\}\in T_n} \bigl( L(X_{abc}) + L(X'_{abc})\big)= 
\sum_{\{a,b,c\}\in T_n} 
\bigl( l(Y_{abc}(M)) + l(Y_{abc}(M'))\bigr)
= \mathrm{const} .
\end{align}

In order to determine the value of the constant in (\ref{sumXX}), we 
take $M=M_x$. Recall that
$Y_{abc}(M'_x) = Y_{abc}(M_x)$ (cf.\ the proof of Lemma~\ref{M0}). Therefore, 
\begin{align}
\nonumber
{}& \sum_{\{a,b,c\}\in T_n} 
\bigl( l(Y_{abc}(M_x)) + l(Y_{abc}(M_x'))\bigr)
 = \sum_{\{a,b,c\}\in T_n} 
\bigl( l(Y_{abc}(M_x)) + l(Y_{a,a+c-b,c}(M_x))\bigr)
 \\
{}& \label{sumXXb}   
 \stackrel{(\ref{Y0})}{=}  \! \sum_{\{a,b,c\}\in T_n} 
\Bigl( l \bigl(\frac{c-b}{b-a}\bigr) + l \bigl(\frac{b-a}{c-b} \bigr) \Bigr)
\stackrel{(\ref{rd1})}{=}  \sum_{\{a,b,c\}\in T_n}  \! 1 = \frac{n(n-1)(n-2)}{6}.
\end{align}

\vspace*{1mm}
In order to prove relation (\ref{dilsum2}), we set    
\begin{align}\label{zy} 
 W_{abc} = \frac{Y^{abc}(M)}{1+Y^{abc}(M)} , \qquad
 W''_{abc} = \frac{Y^{abc}(M'')}{1+Y^{abc}(M'')} .
\end{align}  
Using Pl\"ucker relations (\ref{PluI}) and 
the second relation in (\ref{yty2}), one can check that $W_{abc}$ and 
$W''_{abc}$ coincide, respectively, with $1-X_{abc}$ and $1-X'_{abc}$
if every right flag minor is replaced with its upper counterpart
(e.g. $\Delta_{[a,b]\cup[c+1,n]}$ is replaced with $\Delta^{[a,b]\cup[c+1,n]}$ etc.).
Since the Pl\"ucker relations for right and upper minors are identical, cf. (\ref{PluJ})
and (\ref{PluI}), relation
\begin{align}\label{sumZ}
 \sum_{\{a,b,c\}\in T_n } W_{abc} \wedge (1 - W_{abc}) + 
 \sum_{\{a,b,c\}\in T_n } W''_{abc} \wedge (1 - W''_{abc})
   =0  
\end{align} 
can be proved by repeating the proof of (\ref{sumX}). 
Therefore, by Lemma~\ref{ZD}, we have 
\begin{align}\label{sumZZ}
\sum_{\{a,b,c\}\in T_n} \bigl( L(W_{abc}) + L(W''_{abc})\big)= 
\sum_{\{a,b,c\}\in T_n} 
\bigl( l(Y^{abc}(M)) + l(Y^{abc}(M''))\bigr)
= \mathrm{const} .
\end{align}
 
Substituting $M=M_x$ in (\ref{sumZZ}), using Lemma~\ref{M0},
and repeating the computation (\ref{sumXXb}), we conclude that the constants 
on the r.h.s. of (\ref{sumXX}) and (\ref{sumZZ}) coincide.
\end{proof}

\begin{proof}[\bf Proof of  Proposition~\ref{ZZ2}]  
Combining the first relation in (\ref{yty2}) with  (\ref{dilsum1}), we obtain
\begin{align}\label{pp1}
  \sum\limits_{ \{a,b,c\}\in T_n } l\bigl( Y_{abc}(M) \bigr)
  \stackrel{(\ref{rd1}),(\ref{dilsum1})}{=} 
  \sum\limits_{ \{a,b,c\}\in T_n } l\bigl( 1/Y_{abc}(M') \bigr)
 \stackrel{(\ref{yty2})}{=} \sum\limits_{ \{a,b,c\}\in T_n } l\bigl( \tilde{Y}_{abc}(M) \bigr)  .
\end{align}  
Analogously, combining the second relation in (\ref{yty2}) with  (\ref{dilsum2}), we obtain
\begin{align}\label{pp2}
  \sum\limits_{ \{a,b,c\}\in T_n } l\bigl( \tilde{Y}^{abc}(M) \bigr)
  \stackrel{(\ref{yty2})}{=} 
  \sum\limits_{ \{a,b,c\}\in T_n } l\bigl( 1/Y^{abc}(M'') \bigr)
 \stackrel{(\ref{rd1}),(\ref{dilsum2})}{=} \sum\limits_{ \{a,b,c\}\in T_n } l\bigl( Y^{abc}(M) \bigr)  .
\end{align}  
The l.h.s. of (\ref{pp1}) and (\ref{pp2}) coincide due to (\ref{yty1}), 
so the chain of equalities in (\ref{yaux}) follows.
\end{proof}

\begin{proof}[\bf Proof of  Theorem~\ref{DilGG}]  
Let $Z_{abc}$ be any of the functions 
$Y_{abc}$, $\tilde{Y}_{abc}$, $Y^{abc}$, $\tilde{Y}^{abc}$.
Equations (\ref{YYa}) and (\ref{YYb}) define $Z_{abc}(G)$ for
$G \in B_n^+$. Let $G_1,G_2 \in B_n^+$ be such that
$G_2 = \Lambda_1 G_1 \Lambda_2$, where $\Lambda_1$ and 
$\Lambda_2$ are diagonal matrices with positive diagonal entries. 
Then we have 
\begin{align}\label{dgd}
\Delta_{I} (G_2)  = \Delta_{I} (\Lambda_1)
\Delta_{I} (G_1) \, \Delta_{[n+1-|I|,n]} (\Lambda_2) , \quad 
\Delta^{J} (G_2)  = \Delta^{[1,|J|]} (\Lambda_1)
\Delta^{J} (G_1) \,  \Delta^{J}(\Lambda_2) .
\end{align}
Consider the case  $Z_{abc}=Y_{abc}$. We have
\begin{align}\label{yI}
 Y_{abc}(G_2) = 
\frac{ \Delta_{I_{abc} \cup \{a\}}(G_2) \Delta_{I_{abc} \cup \{b,c\}}(G_2) }
  {     \Delta_{I_{abc} \cup \{a,b\}}(G_2) \Delta_{I_{abc} \cup \{c\}}(G_2) } ,
\end{align}
where $I_{abc}= [a+1,b-1]\cup [c+1,n]$. Using the first relation
in (\ref{dgd}), it is easy to check that the 
contributions from $\Lambda_1$ and $\Lambda_2$ cancel. Therefore, we have
$Y_{abc} (G_1) =   Y_{abc} (G_2)$.  The other cases of $Z_{abc}$ can be 
treated similarly and we conclude that
\begin{align}\label{yyyyg1}
  Z_{abc} (G_1) =    Z_{abc} (G_2) .
\end{align}  

Let $G \in B^+_n$ be given by $G=\Lambda_1 M_1 = M_2 \Lambda_2$, where
$M_1, M_2 \in N^+_n$ and $\Lambda_1, \Lambda_2$ are diagonal matrices.
Using Lemma~\ref{GMB}, Theorem~\ref{Dil1234}, and relations (\ref{yyyyg1}),
we obtain
\begin{align}
\nonumber
& \sum\limits_{ \{a,b,c\}\in T_n } l\Bigl( Z_{abc} (\check{G}) \Bigr)  
 \stackrel{(\ref{gc}), (\ref{yyyyg1})}{=}
\sum\limits_{ \{a,b,c\}\in T_n } l\Bigl( Z_{abc} (M'_2) \Bigr) 
 \stackrel{(\ref{ysig})}{=}
   \sum\limits_{ \{a,b,c\}\in T_n } l\Bigl( \frac{1}{Z_{abc} (M_2) } \Bigr) \\
\label{dilgmg}
&
   \stackrel{(\ref{yyyyg1})}{=}
   \sum\limits_{ \{a,b,c\}\in T_n } l\Bigl( \frac{1}{Z_{abc} (G) } \Bigr)
   \stackrel{(\ref{yyyyg1})}{=}
   \sum\limits_{ \{a,b,c\}\in T_n } l\Bigl( \frac{1}{Z_{abc} (M_1) } \Bigr)  \\
\nonumber
& \stackrel{(\ref{ysig})}{=}
   \sum\limits_{ \{a,b,c\}\in T_n } l\Bigl( Z_{abc} (M''_1) \Bigr) 
   \stackrel{(\ref{gh}), (\ref{yyyyg1})}{=}
    \sum\limits_{ \{a,b,c\}\in T_n } l\Bigl( Z_{abc} (\hat{G}) \Bigr) .
\end{align}  
To complete the proof, it suffices to notice that,
by Proposition~\ref{ZZ2}, the value of the sum on the r.h.s. of (\ref{dilgmg})
is the same for all the  four choices of the function $Z_{abc}$. 
\end{proof}

\vspace*{1mm}
\small{
{\bf Acknowledgements.} 
The authors are grateful to the anonymous referee of 
the journal Linear Algebra and Its Applications  for useful remarks.
This work was supported by the project MODFLAT of the 
European Research Council (ERC) and the National Center of Competence
 in Research SwissMAP of the Swiss National Science Foundation,
 and in part by the Russian Fund for Basic Research Grant No. 18--01--00271.
}

\vspace*{1.5mm}
{\sc \small
\noindent
Section of Mathematics, University of Geneva, 
C.P. 64, 1211 Gen\`eve 4, Switzerland \\[1mm]
Steklov Mathematical Institute,
Russian Academy of Sciences,
Fontanka 27, 191023, St. Petersburg, Russia}

\end{document}